\documentclass[a4paper,11pt,twoside]{article}
\usepackage{ucs}
\usepackage[utf8x]{inputenc}
\usepackage[T1]{fontenc}
\usepackage{lmodern}
\usepackage{amsmath}
\usepackage{amsthm}
\usepackage{mathrsfs}
\usepackage{amsfonts}
\usepackage{array}
\usepackage[all]{xy}
\usepackage{fancyhdr}
\usepackage{graphicx}
\usepackage{epstopdf}
\usepackage{bbm}
\usepackage{amssymb}

\author{Hubert \sc{Klaja} \thanks{Laboratoire Paul Painlev\' e,
Universit\'e Lille 1, CNRS UMR 8524, B\^at. M2,
F-59655 Villeneuve \newline d'Ascq Cedex, France ; {\tt hubert.klaja@gmail.com}}} %
\title{Hyperinvariant subspaces for some compact perturbations of multiplication operators}
\date{}

\theoremstyle{definition} \newtheorem{definition}{Definition}[section]
\theoremstyle{definition} \newtheorem{example}[definition]{Example}
\theoremstyle{remark}     
\theoremstyle{plain}      \newtheorem{theorem}[definition]{Theorem}
\theoremstyle{plain}      \newtheorem{proposition}[definition]{Proposition}
\theoremstyle{plain}      \newtheorem{corollary}[definition]{Corollary}
\theoremstyle{plain}      \newtheorem{lemma}[definition]{Lemma}
\theoremstyle{definition}

  \newcommand{\norme}[1]
    {\left\| #1 \right\|}
  \newcommand{\abs}[1]
    {\left\vert #1 \right\vert}
  \newcommand{\ps}[2]
    {\left\langle #1,#2 \right\rangle}
  \newcommand{\integral}[3]
    {\int_{#1} #2 \mathrm{d}#3}
  \newcommand{\ind}[1]
    {\mathbbm{1}_{#1}}
  \newcommand{\conj}[1]
    {\overline{#1}}
  \newcommand{\Range}
    {\mathrm{Ran}}
      
  \newcommand{\I}{\mathfrak{i}}

  \newcommand{\B}[1]{\mathcal{B}(#1)}
  \newcommand{\K}[1]{\mathcal{K}(#1)}
  \newcommand{\N}{\mathbb{N}}
  
  \newcommand{\R}{\mathbb{R}}
  \newcommand{\C}{\mathbb{C}}
  
  \newcommand{\Lun}{L_1}
  \newcommand{\Lp}[1]{L_{#1}}
  
  \newcommand{\Spec}[1]{\sigma(#1)} 
  \newcommand{\Specp}[1]{\sigma_p(#1)}
  \newcommand{\Spece}[1]{\sigma_e(#1)}
   
  \newcommand{\Com}[1]{\lbrace #1 \rbrace'}
  
  \newcommand{\Reel}[1]{\mathsf{Re}(#1)}
  
  \newcommand{\tens}[2]{#1 \otimes #2}
  \newcommand{\m}{m}

\begin{document}

\maketitle

\begin{abstract} 
  In this paper, a sufficient condition for the existence of hyperinvariant subspace of compact perturbations of multiplication operators on some Banach spaces is presented. An interpretation of this result for compact perturbations of normal and diagonal operators on Hilbert space is also discussed. An improvement of a result of \cite{Fang_Xia_2012} for compact perturbations of diagonal operators is also obtained.
\end{abstract}

\section{Introduction}
Let $X$ be a separable complex Banach space. The invariant subspace problem is the question whether every bounded linear operator $T \in \B{X} $ has a non trivial invariant subspace; in other words does there exist a closed subspace $M$ of $X$ such that $M \ne \lbrace 0 \rbrace$, $M \ne X $ and $T(M) \subset M $? The hyperinvariant subspace problem is the question whether every bounded linear operator $T \in \B{X}$ such that $T \ne \lambda I $ has a non trivial hyperinvariant subspace, i.e. whether there exists a closed subspace $M$ of $X$ such that $M \ne \lbrace 0 \rbrace$, $M \ne X $ and for every bounded operator $S \in \B{X}$ such that $ST = TS $, we have  $S(M) \subset M $? Enflo \cite{Enflo_1987} and Read \cite{Read_1986} proved that the invariant subspace problem has a negative answer on some Banach spaces. On the other hand, Argyros and Haydon \cite{Argyros_Haydon_2011} constructed a Banach space where every bounded linear is a compact perturbation of a scalar operator, hence by Lomonosov's celebrated result \cite{Lomonosov_1973}, every non scalar operator has a non trivial hyperinvariant subspace. However the invariant and hyperinvariant subspace problem are still open in reflexive Banach spaces, and in particular in Hilbert spaces. For normal operators in Hilbert spaces, the spectral theorem ensures the existence of an hyperinvariant subspace. Lomonosov \cite[Theorem 6.1.2]{Chalendar_Partington_2011} proved that every compact operator on a Banach space has a non trivial invariant subspace. But if $N$ is a normal operator on a Hilbert space $H$, and $K$ is compact operator on $H$, we don't know in general if $N+K$ has a non trivial hyperinvariant subspace or not. We refer the reader to the book \cite{Chalendar_Partington_2011} for more information about the Invariant Subspace Problem.

  In 2007 Foias, Jung, Ko and Pearcy \cite{Foias_Jung_Ko_Pearcy_2007} proved the following theorem. 
\begin{theorem}[\cite{Foias_Jung_Ko_Pearcy_2007}]
  Let $(e_n)_{n \in \N} $ be an orthonormal basis in a separable complex Hilbert space $H$. Let $D = \sum_{n \in \N} \lambda_n \tens{e_n}{e_n} $ be a bounded diagonal operator on $H$. Let $u,v \in H $ be two vectors. If
  $$
  \sum_{n \in \N} \abs{\ps{u}{e_n}}^{\frac{2}{3}} < \infty, \quad \sum_{n \in \N} \abs{\ps{v}{e_n}}^{\frac{2}{3}} < \infty,
  $$
  and if $D + \tens{u}{v} \ne \lambda I $, then the rank one perturbation $D + \tens{u}{v} $ of the diagonal operator $D$ has a non trivial hyperinvariant subspace.
\end{theorem}
  In 2012 Fang and Xia \cite{Fang_Xia_2012} improved this result. Their approach allowed to consider finite rank perturbations of a diagonal operator. They also improved the summability condition of Foias, Jung, Ko and Pearcy. Here is their result.
\begin{theorem}[\cite{Fang_Xia_2012}]
  \label{ThFanXia2012}
  Let $(e_n)_{n \in \N} $ be an orthonormal basis in a separable complex Hilbert space $H$. Let $D = \sum_{n \in \N} \lambda_n \tens{e_n}{e_n} $ be a bounded diagonal operator on $H$. Let $u_1, \dots, u_r, v_1, \dots, v_r \in H $ be vectors. If
  $$
  \sum_{k = 1}^r \sum_{n \in \N} \abs{\ps{u_k}{e_n}} < \infty, \quad \sum_{k = 1}^r \sum_{n \in \N} \abs{\ps{v_k}{e_n}} < \infty,
  $$
  and if $D + \sum_{i=1}^r \tens{u_i}{v_i} \ne \lambda I $, then the finite rank perturbation $D + \sum_{i=1}^r \tens{u_i}{v_i} $ of the diagonal operator $D$ has a non trivial hyperinvariant subspace.  
\end{theorem}
  
  The goal of this paper is to improve Fang and Xia's approach in order to deal with some compact perturbations of multiplication operators on separable $\Lp{p} $ spaces. The well-known spectral theorem for normal operator tells us that every normal operator is a multiplication operator on some $\Lp{2} $ space. As a diagonal operator is a particular case of a normal operator, this can be seen as a generalization of the previous result.

\subsection{Notations}
  In this paper, we will denote by $H$ a separable complex Hilbert space, and by $X$ a separable complex Banach space. We will denote by $m $ the Lebesgue measure on the complex plane. We will denote the set of all bounded operators (respectively the set of all compact operators) acting on $X$ by $\B{X}$ (respectively $\K{X} $). Let $T \in \B{X}$ be a bounded operator. We will denote the commutant of $T$ by
  $$
    \Com{T}= \lbrace S \in \B{X}, ST = TS \rbrace.
  $$
  We will also denote respectively the spectrum, the point spectrum and the essential spectrum of an operator $T$ by $\Spec{T}$, $\Specp{T} $ and $\Spece{T} $. Let $ (\Omega,\mu) $ be a borelian $\sigma$-finite measure space. Let $p,q \in ]1, \infty[ $ be two positive numbers such that $\frac{1}{p} + \frac{1}{q} =1 $. If $f \in \Lp{\infty}(\Omega,\mu)$ is a bounded complex valued function, we will denote by $M_f : \Lp{p}(\Omega,\mu) \rightarrow \Lp{p}(\Omega,\mu)$ the linear operator defined by $M_f(g)(\xi) = f(\xi)g(\xi) $.
  
   Let $(s_n)_{n \in \N}$ be a sequence of positive real numbers such that $\lim_{n \rightarrow \infty} s_n = 0 $. Let $(u_n)_{n \in \N} $ be a sequence in $ \Lp{p}(\Omega,\mu)$ and $(v_n)_{n \in \N} $ be a sequence in $ \Lp{q}(\Omega,\mu)$. For all $u,x \in \Lp{p}(\Omega,\mu)$ and $v \in \Lp{q}(\Omega,\mu)$, we define $\tens{u}{v}(x) = \left(\integral{\Omega}{x(\xi) \conj{v(\xi)}}{\mu(\xi)}\right) u $. This will avoid a change of notation in Hilbert spaces. Indeed, in the case $p=q=2$, we have that $\tens{u}{v}(x) = \ps{x}{v}u $. We will denote by $K:\Lp{p}(\Omega,\mu) \rightarrow \Lp{p}(\Omega,\mu) $ the operator defined by $K = \sum_{n \in \N} s_n \tens{u_n}{v_n} $. In general this operator need not be compact (it may also be unbounded).
  
\subsection{Main Results}
  Here are the main results of the paper. The first result is a generalization of Fang and Xia's approach in \cite{Fang_Xia_2012}. The generalization allows us to consider some compact perturbations of multiplication operators in $\Lp{p}$ spaces. Remember that a diagonal operator is a particular case of a multiplication operator on a $\Lp{2}(\Omega,\mu)$ space with $\mu$ being a purely atomic measure.
  
\begin{theorem}
  \label{ThMainTheorem}
  Let $ (\Omega,\mu) $ be a borelian $\sigma$-finite measure space. Let $f \in \Lp{\infty}(\Omega,\mu) $ be a bounded complex valued function. 
Let $(u_n)_{n \in \N} $ be a sequence in $ \Lp{p}(\Omega,\mu)$ and $(v_n)_{n \in \N} $ be a sequence in $ \Lp{q}(\Omega,\mu)$.
  Denote by $K$ the operator defined by $K = \sum_{n \in \N} s_n \tens{u_n}{v_n} $. Suppose that $K$ is compact and that there exists a rectifiable piecewise smooth Jordan curve $\Gamma$ in $\C$ such that
  \begin{itemize}
    \item[1.] There exist $a,b \in \Spece{M_f}$ such that $a$ is in the interior of $\Gamma$ and $b$ is in the exterior of $\Gamma$,
    \item[2.] $\mu(f^{-1}(\Gamma))=0$,
    \item[3.] For all $n \in \N$, $z \in \Gamma$, we have that $u_n \in \Range(M_f -z)$ and $v_n \in \Range(M_f-z)^* $,
    \item[4.] Denote by $A(z)$ the (possibly unbounded) operator $A(z) = \sum_{n \in \N} s_n \tens{ \left( (M_f-z)^{-1}u_n \right) }{\left( (M_{\conj{f}} -\conj{z})^{-1} v_n \right)}$. For all $z \in \Gamma$, we suppose that $A(z)$ is a compact operator, and $A : \Gamma \rightarrow \K{\Lp{p}(\Omega,\mu)} $ is a continuous application. 
  \end{itemize}
  Then the bounded operator $T = M_f + K$ acting on $\Lp{p}(\Omega,\mu)$ has a non trivial hyperinvariant subspace.
\end{theorem}

Note that if $T$ satisfies the assumptions of Theorem \ref{ThMainTheorem}, then $\Spece{T} = \Spece{M_f + K} = \Spece{M_f}$. As $M_f$ has two distinct values in its essential spectrum, $T$ also has. Hence $T$ can not be a scalar operator. 
The second result is a generalization of Fang and Xia's result (cf Theorem \ref{ThFanXia2012}) in the particular case of compact perturbation of a diagonal operator on Hilbert spaces. This is a consequence of the previous Theorem.

\begin{theorem}
    \label{ThDiagPlusKompactSummabilityCondition}
  Let $(e_k)_{k \in \N} $ be an orthonormal basis of $H$. Let $D = \sum_{k \in \N} \lambda_k \tens{e_k}{e_k} $ be a bounded diagonal operator on a Hilbert space. Let $ K = \sum_{n \in \N} s_n \tens{u_n}{v_n} $ be a compact operator.
  If there exist two sequences $(a_n)_{n \in \N}, (b_n)_{n \in \N}$ such that for all $n \in \N$, $a_n b_n = s_n $ and
  \begin{align}
    \sum_{n \in \N} \sum_{k \in \N} \abs{a_n \ps{u_n}{e_k}} < \infty  \label{Cond1} \\
    \sum_{n \in \N} \sum_{j \in \N} \abs{b_n \ps{e_j}{v_n}} < \infty, \label{Cond2}
  \end{align}
  and if $D+K \ne \lambda I $, then $T=D+K$ has a non-trivial hyperinvariant subspace.
\end{theorem}
  
  Of course, Theorem \ref{ThFanXia2012} is contained in this one.

\subsection{Preliminaries}

  Before we start the proof of the mains theorem, we will need some material.
  Our first statement is a folklore result. A proof of it in the Hilbert space case using Lomonosov's Theorem can be found in \cite[Proposition 4.1]{Fang_Xia_2012}.
\begin{proposition}
  \label{PropPplusKaSHI}
  Let $P \in \B{X}$ be an idempotent such that $\dim(P(X)) = \dim((I-P)(X)) = \infty $. Then for any compact operator $L$, the operator $P+L$ has a non-trivial hyperinvariant subspace.
\end{proposition}
  
\begin{proof}
  First, note that if $\Specp{P+L} \ne \emptyset $, then $P+L$ has a non trivial hyperinvariant subspace. Suppose that $\Specp{P+L} = \emptyset $. By Weyl's Theorem (see for instance \cite[Chapter 0, Theorem 0.10]{Radjavi_Rosenthal_1973}), we have that $\Spec{P+L} \subset \Spec{P} \cup \Specp{P+L} = \Spec{P} = \lbrace 0,1 \rbrace$. As $\lbrace 0,1 \rbrace = \Spece{P} \subset \Spec{P+L}$, we get that $\Spec{P+L} = \lbrace 0,1 \rbrace $. So by the Riesz-Dunford functional calculus, we infer that $P+L$ has a non trivial hyperinvariant subspace.
\end{proof}

  The next statement is a well known fact. The reader can find a proof in \cite{Hille_1948}.

  \begin{proposition}
  \label{PropIntegralOfKompactOperatorsIsCompact}
  Let $\Gamma$ be a rectifiable piecewise smooth Jordan curve. If $F: \Gamma \rightarrow \K{X}$ is a continuous application then 
  $$
  L = \integral{\Omega}{F(z)}{z}
  $$
  exists and is a compact operator.
\end{proposition}
  

  We recall next a well known result concerning normal operators on complex Hilbert spaces. Its states that every normal operator on an Hilbert space can be seen as a multiplication operator on some measure space. We refer the reader to \cite[Theorem 2.4.5]{Arveson_2002}, for a proof of this result. 
  
\begin{theorem}
  \label{ThSpectralTheorem}
  Let $N \in \B{H} $ be a normal operator on a complex Hilbert space $H$. Then there exists a sigma-finite measure space $(\Omega, \mu)$, a bounded function $f \in \Lp{\infty}(\Omega,\mu) $ and a unitary operator $W: H \rightarrow \Lp{2}(\Omega,\mu) $ such that
  $$
     M_f W =  W N.
  $$  
\end{theorem}

  Lastly we mention a well known result for compact operators on a Hilbert space. The reader can find a proof of this result in \cite[Chapter VI, Theorem 1.1]{Gohberg_Goldberg_Kaashoek_1990}. 
  
\begin{theorem}
   \label{ThSingularValueDecomposition}
   Let $K \in \K{H} $ be a compact operator on the Hilbert space $H$. Then there exist two orthonormal families $(u_n)_{n \in \N}, (v_n)_{n \in \N} $ of vectors in $H$ and a sequence $(s_n)_{n \in \N}$ of positive real numbers such that $\lim_{n \rightarrow \infty} s_n = 0$, and
  $$
    K = \sum_{n \in \N} s_n \tens{u_n}{v_n}.
  $$
\end{theorem}

\section{Proof of Theorem \ref{ThMainTheorem}}



  To prove Theorem \ref{ThMainTheorem}, we will use the same approach as in \cite{Fang_Xia_2012}. The idea is to create, for all $z \in \Gamma$, a "nice" right inversion formula for $T-z$. Then, using some unconventional Riesz-Dunford functional calculus, we will prove that the commutant of $T$ is included in the commutant of a compact perturbation of an idempotent. This last operator will have a non trivial hyperinvariant subspace, and so $T$ will as well.
  We start with some technical results for building the right inversion formula. In this section we will assume that the assumptions of Theorem \ref{ThMainTheorem} are always satisfied. In particular we need to assume that $K= \sum_{n \in \N} s_n \tens{u_n}{v_n}$ is a compact operator (as it is written, $K$ need not  be a compact operator in general). 
  
\begin{lemma}
  \label{LemmaNoEigenvalueImpliesInvertibility}
  Denote by $T= M_f + K$ the compact perturbation of the multiplication operator $M_f$ on the Banach space $\Lp{p}(\Omega,\mu)$. Suppose that assumptions 3 and 4 of Theorem \ref{ThMainTheorem} are satisfied and $\Specp{T} \cap \Gamma = \emptyset $. Then for every $z \in \Gamma$, $I + A(z)(M_f -z) $ is invertible.
\end{lemma}

\begin{proof}
  Suppose that for some $z \in \Gamma$, $I + A(z)(M_f -z) $ is not invertible. As $A(z)$ is compact and $M_f -z $ is a bounded operator, we have that $ A(z)(M_f -z) $ is compact. So $-1 \in \Specp{A(z)(M_f -z)} $. Hence there exists $h \in \Lp{p}(\Omega,\mu) $ such that $h \ne 0 $ and $A(z)(M_f -z)h = -h $. We have that
\begin{align*}
  -h &= A(z)(M_f -z) h \\
     &= \left( \sum_{n \in \N} s_n \tens{ \left( (M_f-z)^{-1}u_n \right) }{\left( (M_{\conj{f}} -\conj{z})^{-1} v_n \right)} \right) (M_f -z) h \\
     &= \left( \sum_{n \in \N} s_n \tens{ \left( (M_f-z)^{-1}u_n \right) }{ v_n } \right) h.
\end{align*}
  Applying $ (M_f-z) $ on each side of the equality, we obtain 
  $$
  -(M_f-z)h =  \left( \sum_{n \in \N} s_n \tens{ u_n }{ v_n } \right) h = K h. 
  $$
  So we have that $zh = (M_f + K) h = Th $, thus $z \in \Specp{T} \cap \Gamma $ which is a contradiction with the assumption that $ \Specp{T} \cap \Gamma = \emptyset $.
\end{proof}

  The following lemma is a straightforward corollary of Lemma \ref{LemmaNoEigenvalueImpliesInvertibility}.

\begin{lemma}
  \label{LemBIsAContinousMapToCompactOperator}
  Suppose that assumptions 3 and 4 of Theorem \ref{ThMainTheorem} are satisfied and $ \Specp{T} = \emptyset $. Then for all $ z \in \Gamma $, $B(z) = \big( I + A(z)(M_f -z) \big)^{-1} A(z) $ is a compact operator. Moreover the application
  \begin{align*}
    B: \Gamma &\rightarrow \K{\Lp{p}(\Omega,\mu)} \\
        z & \mapsto B(z)
  \end{align*}
  is continuous.
\end{lemma}


  
  Our next lemma is 
  
\begin{lemma}
  \label{LemDensity}
  Let $\Gamma$ be a rectifiable piecewise smooth Jordan curve such that assumption 2 of Theorem \ref{ThMainTheorem} is satisfied. Let $\mathcal{L} \subset \Lp{p}(\Omega,\mu)$ be the linear manifold of all finite linear combination of indicator functions of measurable sets $S_i $ such that $f(S_i)$ is at a strictly positive distance of $\Gamma$. Let $W= \cap_{z \in \Gamma} \Range(M_f-z)$. Then $\mathcal{L}$ and $W$ are dense in $\Lp{p}(\Omega, \mu)$.
\end{lemma}  
  
\begin{proof}
   We have that $w \in \mathcal{L}$ if and only if there exist $a_1, \dots, a_r \in \C $ and $S_1, \dots, S_r$ measurable subsets of $\Omega$ such that $w= \sum_{i=1}^r a_i \ind{S_i}$ and $\inf_{\xi \in S_i, z \in \Gamma} \abs{f(\xi) - z} > 0 $ for each $i=1, \dots, r$.

  In order to prove that the closure of $\mathcal{L}$ is $\Lp{p}(\Omega, \mu) $, we just need to prove that all indicator function of measurable sets are in the closure of $\mathcal{L}$,  because the linear manifold of all finite linear combination of indicator function is dense in $ \Lp{p}(\Omega, \mu) $. Let $B$ a measurable subset of $ \Omega $ and denote by $ B_\varepsilon = \lbrace \xi \in B, \textrm{dist}(f(\xi),\Gamma) > \varepsilon \rbrace $. We have that $ \ind{B_\varepsilon} $ goes to $\ind{B} $ as $\varepsilon$ goes to $0$ (because $\mu(f^{-1}(\Gamma))=0$) and $ \ind{B_\varepsilon} \in \mathcal{L} $. 
  
   Then the closure of $\mathcal{L}$ is $\Lp{p}(\Omega, \mu) $. As $\mathcal{L} \subset W $, the closure of $W $ is $\Lp{p}(\Omega, \mu)$ as well.
\end{proof}  
  
  Next comes the following analogue of Lemma 3.4 of \cite{Fang_Xia_2012}.  
  
\begin{lemma}
  \label{LemmaRightInversionFormula}
   With the notations of Lemma \ref{LemBIsAContinousMapToCompactOperator}, for all $z \in \Gamma$, denote by $R(z)$ the (possibly unbounded) operator defined by $R(z) = (M_f - z)^{-1} - B(z) $. Then for every $w \in W$ we have that
   $$
   (T-z)R(z)w = w.
   $$
\end{lemma}

 In this lemma, $R(z)$ can be an unbounded operator because $(M_f - z)^{-1} $ can be unbounded if $z \in  \Spec{M_f} \cap \Gamma $. According to Lemma \ref{LemBIsAContinousMapToCompactOperator}, $B(z)$ is a compact operator for each $ z \in \Gamma $.

\begin{proof}
  Let $ w \in W$ and $z \in \Gamma $. Observe that
\begin{align*}
  (M_f -z)A(z)(M_f - z) 
    &= (M_f -z)\left( \sum_{n \in \N} s_n \tens{ \left( (M_f-z)^{-1}u_n \right) }{\left( (M_{\conj{f}} -\conj{z})^{-1} v_n \right)} \right)(M_f - z) \\
    &= \sum_{n \in \N} s_n \tens{u_n}{v_n} \\
    &= K.
\end{align*}
  For all $w \in W \subset \Range(M_f - z)$ it makes senses to write $R(z)w$. Replacing $K$ by this expression, we have that
\begin{align*}
  (T-z)R(z)w
    &= (M_f -z + K)\left( (M_f -z)^{-1} - \big( I + A(z)(M_f -z) \big)^{-1} A(z) \right)w \\
    &= (M_f -z)\Big(I + A(z)(M_f -z)\Big)\left( (M_f -z)^{-1} - \big( I + A(z)(M_f -z) \big)^{-1} A(z) \right)w \\
    &= (M_f -z)\left( (M_f -z)^{-1} + A(z) - A(z) \right)w \\
    &= w,
\end{align*}
which proves Lemma \ref{LemmaRightInversionFormula}.
\end{proof}

\begin{lemma}
  \label{LemSWsubsetW}
  Let $S \in \Com{T}$ and $w \in W$. Then $Sw \in W$. 
\end{lemma}

\begin{proof}
  Let $S \in \Com{T}$, $z \in \Gamma$ and $w \in W$. Using in the fourth equality the fact that $ K = (M_f -z)A(z)(M_f - z) $, we have that
  \begin{align*}
    Sw
      &= S(T-z)R(z)w \\
      &= (T-z)SR(z)w \\
      &= (M_{f}-z)SR(z)w + KSR(z)w \\
      &= (M_{f}-z)SR(z)w + (M_f -z)A(z)(M_f - z)SR(z)w \\
      &= (M_{f}-z)\Big( SR(z)w + A(z)(M_f - z)SR(z)w \Big).
  \end{align*}  
  So $Sw \in \Range(M_f -z) $.
\end{proof}

\begin{proposition}
  \label{PropIntegralOfMultiplicationIsProjector}
  Let $ \Gamma$ satisfy assumptions 1 and 2 of Theorem \ref{ThMainTheorem}. Denote by $\Theta $ the interior of $\Gamma$. Then for all $w \in \mathcal{L} $ we have
  $$
  M_{\ind{f^{-1}(\Theta)}}w = \frac{-1}{2\I\pi}\integral{\Gamma}{(M_f -z)^{-1}w \; }{z},
  $$
  with $\Gamma$ oriented in the counter clockwise direction. 
  Moreover, if there exist $a,b \in \Spece{M_f} $ such that $a \in \Theta $ and $b \notin \Theta \cup \Gamma $, then $ \dim(\Range(M_{\ind{f^{-1}(\Theta)}})) = \dim(\Range(I - M_{\ind{f^{-1}(\Theta)}})) = \infty $. 
\end{proposition}

  Note that $M_{\ind{f^{-1}(\Theta)}} $ is an idempotent (i.e. $ (M_{\ind{f^{-1}(\Theta)}})^2 = M_{\ind{f^{-1}(\Theta)}} $).

\begin{proof}
  Let $w \in \mathcal{L}$. So there exist $a_1, \dots, a_r \in \C $ and $S_1, \dots, S_r$ measurable subsets of $\Omega$ such that $w= \sum_{i=1}^r a_i \ind{S_i}$ and $\inf_{\xi \in S_i, z \in \Gamma} \abs{f(\xi) - z} > 0 $ for each $i=1, \dots, r$. As $\mu(f^{-1}(\Gamma))=0 $, 
  we have for $\mu$-almost every $\xi \in \Omega $ that $f(\xi) \notin \Gamma $ and
\begin{align*}
   \frac{1}{2\I\pi}\integral{\Gamma}{(M_f -z)^{-1}w(\xi)}{z}
     &= \sum_{i=1}^r \frac{1}{2\I\pi} \integral{\Gamma}{\frac{a_i \ind{S_i}(\xi)}{f(\xi) - z}}{z} \\
     &= \sum_{i=1}^r a_i \ind{S_i}(\xi) \frac{1}{2\I\pi} \integral{\Gamma}{\frac{1}{f(\xi) - z}}{z} \\
     &= - \sum_{i=1}^r a_i \ind{S_i}(\xi) \ind{\Theta}(f(\xi)) \\
     &= - M_{\ind{f^{-1}(\Theta)}} w (\xi).
\end{align*}

  Now we will prove that $a \in \Spece{M_f} \cap \Theta $ implies that $\dim(\Range(M_{\ind{f^{-1}(\Theta)}})) = \infty $. A similar argument works for the other assertion. First note that for every compact operator $L \in \K{ \Lp{p}(\Omega,\mu)} $, we have $a \in \Spec{M_f + L} $. In other words, $M_f + L - aI $ does not have a bounded inverse. Fix $\varepsilon > 0 $ and denote by $B$ the disk $B = \lbrace w \in \C, \abs{a - w} < \varepsilon \rbrace $. Denote by $\tilde{f} = f - (f - a - \varepsilon)\ind{f^{-1}(B)} $. If $\abs{f(\xi)-a} \ge \varepsilon $, then $\tilde{f}(\xi) - a  = f(\xi) - a $. Otherwise $\tilde{f}(\xi) - a  = \varepsilon $. Now $ \tilde{f}$ is a bounded function and $\tilde{f}-a $ is bounded away from zero (i.e. there exists a constant $c>0$ such that for almost every $\xi \in \Omega$, $\abs{\tilde{f}(\xi) - a} \ge c > 0  $). So $\frac{1}{\tilde{f} - a} $ is a bounded function and 
  $$
  M_{\frac{1}{\widetilde{f} - a}} = (M_{\tilde{f}} -a)^{-1} = (M_f-M_{f-a-\varepsilon} M_{\ind{f^{-1}(B)}} -a )^{-1}
  $$ 
  is a bounded operator. If $M_{\ind{f^{-1}(B)}}$ were a compact operator then $M_{\tilde{f}} -a $ would not be invertible. So  $M_{\ind{f^{-1}(B)}}$ is not a compact idempotent and $\dim(\Range(M_{\ind{f^{-1}(B)}})) = \infty $. If we choose $\varepsilon $ small enough we have that  $ \Range(M_{\ind{f^{-1}(B)}}) \subset \Range(M_{\ind{f^{-1}(\Theta)}}) $, so $\dim(\Range(M_{\ind{f^{-1}(\Theta)}})) = \infty $.
\end{proof}

\begin{proof}[Proof of Theorem \ref{ThMainTheorem}]
  Suppose that $\Specp{T} = \emptyset$. Recall that for all $z \in \Gamma$, $B(z) = \big(I+A(z)(M_f -z)\big)^{-1}A(z)$ and $R(z) =  (M_f -z)^{-1} - B(z)$. Then by Lemma \ref{LemBIsAContinousMapToCompactOperator}, $B(z)$ is a compact operator and the application $B: \Gamma \rightarrow \K{X}$ is continuous. So $\norme{B(z)}$ is bounded on the compact set $\Gamma$ and we have
  $$
     \integral{\Gamma}{\norme{B(z)}}{z} < \infty.
  $$
   Moreover, by Lemma \ref{LemmaRightInversionFormula}, we have for all $w \in W$ that $(T-z)R(z)w=w $. 
  From Proposition \ref{PropIntegralOfKompactOperatorsIsCompact}, we have that 
  $$
  L = \frac{1}{2\I\pi}\integral{\Gamma}{B(z)}{z} 
  $$
  is a compact operator. From Proposition \ref{PropIntegralOfMultiplicationIsProjector}, we know that there exists an idempotent $P$ ($P = M_{\ind{f^{-1}(\Theta)}}$) such that for all $w \in \mathcal{L} $, 
  $$
  Pw = \frac{-1}{2\I\pi}\integral{\Gamma}{(M_f -z)^{-1}w}{z},
  $$
  and such that $\dim(P(X)) = \dim((I-P)(X)) = \infty $.
  
  Let $S \in \Com{T}$. Then for all $w \in W$ we have that
  $
  (T-z)SR(z)w = S(T-z)R(z)w = Sw = (T-z)R(z)Sw
  $
  (because $Sw \in W $ by Lemma \ref{LemSWsubsetW}). As $\Specp{T} = \emptyset$, $T-z$ is injective so $ SR(z)w = R(z)Sw $. Then for all $w \in \mathcal{L}$ (remember that $\mathcal{L} \subset W $) we have
  $$
  S(P+L)w =  \frac{-1}{2\I\pi}\integral{\Gamma}{SR(z)w \; }{z} 
          = \frac{-1}{2\I\pi}\integral{\Gamma}{R(z)Sw \;}{z} 
          = (P+L)Sw.
  $$
  As the closure of $\mathcal{L}$ is $\Lp{p}(\Omega,\mu)$, we get that $S \in \Com{P+L} $. So $\Com{T} \subset \Com{P+L}$. As $P+L $ has a non trivial hyperinvariant subspace by Proposition \ref{PropPplusKaSHI}, $T$ also has one.
\end{proof}

  Let $N \in \B{H} $ be a normal operator on a Hilbert space. Let $(\Omega,\mu)$ be a measure space, $f \in \Lp{\infty}(\Omega,\mu)$ and $W: \Lp{2}(\Omega,\mu) \rightarrow H $ be a unitary operator satisfying the consequences of Theorem \ref{ThSpectralTheorem}. Let $K \in \K{H} $ be a compact operator. Then $WKW^* $ is a compact operator on $\Lp{2}(\Omega,\mu)$, so by Theorem \ref{ThSingularValueDecomposition} there exist a sequence $(s_n)_{n \in \N}$ of positive real numbers such that $\lim_{n \rightarrow \infty} s_n = 0$ and two orthonormal families $(u_n)_{n \in \N}, (v_n)_{n \in \N} $ of vectors in $H$ such that
  $
    W K W^* = \sum_{n \in \N} s_n \tens{u_n}{v_n}
  $.
With these notations, one can state a direct corollary of Theorem \ref{ThMainTheorem} for compact perturbations of normal operators on Hilbert spaces.

\begin{corollary}
  \label{CoroNormalPlusKompact}
  Let $N \in \B{H}$ be a bounded normal operator and $K \in \K{H} $ be a compact operator. With the notations as above, suppose that there exists a rectifiable piecewise smooth Jordan curve $\Gamma$ such that
  \begin{itemize}
    \item[1.] There exist $a,b \in \Spece{N}$ such that $a$ is in the interior of $\Gamma$ and $b$ is in the exterior of $\Gamma$,
    \item[2.] $\mu(f^{-1}(\Gamma))=0$,
    \item[3.] For all $n \in \N$, $z \in \Gamma$, we have that $u_n \in \Range(M_f -z)$ and $v_n \in \Range(M_f-z)^* $,
    \item[4.] Denote by $A(z)$ the (possibly unbounded) operator $A(z) = \sum_{n \in \N} s_n \tens{ \left( (M_f-z)^{-1}u_n \right) }{\left( (M_{\conj{f}} -\conj{z})^{-1} v_n \right)}$. For all $z \in \Gamma$, we suppose that $A(z)$ is a compact operator, and $A : \Gamma \rightarrow \K{H} $ is a continuous application. 
  \end{itemize}   
  Then the operator $T = N + K$ has a non trivial hyperinvariant subspace.
\end{corollary}

We next give some simple applications of this corollary

\begin{example}
  Let $ (\Omega,\mu) $ be a borelian $\sigma$-finite measure space. More precisely, we set $\Omega = \lbrace \xi \in \C, \abs{\xi} \le 1 \rbrace $ and we set $\mu = m $ be the Lebesgue measure on the complex plane. Denote by $A= \lbrace \xi \in \C, \frac{1}{3} \le \abs{\xi} \le \frac{2}{3} \rbrace $. Let $f \in \Lp{\infty}(\Omega, \mu)$ be the bounded function defined by $f(\xi)=\xi $. Let $g,h \in \Lp{2}(\Omega,\mu) $, and denote by $u = (1-\ind{A})g $ and $v = (1-\ind{A})h $. Let $\Gamma = \lbrace z \in \C, \abs{z}= \frac{1}{2} \rbrace $. Then $\Spece{M_f}= \Omega $, $\mu(f^{-1}(\Gamma))=0 $ and for all $z \in \Gamma$, $\frac{u}{f-z},\frac{v}{\conj{f-z}} \in \Lp{2}(\Omega,\mu)$. Moreover the application
  \begin{align*}
    A : &\Gamma \rightarrow \K{H} \\
        & z \mapsto \tens{\frac{u}{f-z}}{\frac{v}{\conj{f-z}}}
  \end{align*}
  is continuous. By Corollary \ref{CoroNormalPlusKompact}, $M_f + \tens{u}{v} $ has a non trivial hyperinvariant subspace.
\end{example}

\begin{example}
  Let $ (\Omega,\mu) $ be a borelian $\sigma$-finite measure space. More precisely, we set $\Omega = \lbrace \xi \in \C, \abs{\xi} \le 2 \rbrace $ and we set $\mu = m $ be the Lebesgue measure on the complex plane. Let $f \in \Lp{\infty}(\Omega, \mu)$ be the bounded function defined by $f(\xi)=\xi $. Let $g_n, h_n \in \Lp{2}(\Omega,\mu) $ such that $\norme{g_n} \le 1 $ and  $\norme{h_n} \le 1 $, and denote by $u_n(\xi) = (1-\abs{\xi})g_n(\xi) $ and $v_n(\xi) = (1-\abs{\xi})h_n(\xi) $. Let $(s_n)_{n \in \N} $ be a sequence of positive real numbers such that $\sum_{n \in \N} s_n < \infty $. Let $\Gamma = \lbrace z \in \C, \abs{z}= 1 \rbrace $. Then for all $z \in \Gamma $ we have
\begin{align*}
  \integral{\Omega}{\frac{\abs{u_n(\xi)}^2}{\abs{\xi -z}^2}}{\mu(\xi)}
    \le   \integral{\Omega}{\frac{\abs{1-\abs{\xi}}^2 \abs{g_n(\xi)}^2}{\abs{\abs{\xi} -\abs{z}}^2}}{\mu(\xi)} 
    &= \integral{\Omega}{\frac{\abs{1-\abs{\xi}}^2 \abs{g_n(\xi)}^2}{\abs{\abs{\xi} -1}^2}}{\mu(\xi)} \\ 
    &= \integral{\Omega}{ \abs{g_n(\xi)}^2}{\mu(\xi)} < \infty.
\end{align*}
  So $u_n \in \Range(M_f-z) $. In the same way, we can prove that $v_n \in \Range(M_f-z)^* $. For all $z \in \Gamma$, we have that
\begin{align*}
  \norme{A(z)} 
    &= \norme{ \sum_{n \in \N} s_n \tens{ \left( (M_f-z)^{-1}u_n \right) }{\left( (M_{\conj{f}} -\conj{z})^{-1} v_n \right)}} \\
    &\le \sum_{n \in \N} s_n \norme{g_n} \norme{h_n} \\
    &\le \sum_{n \in \N} s_n  < \infty.
\end{align*}
  So $A(z)$ is a bounded operator. Denote by $A_N(z) =  \sum_{n =1}^N s_n \tens{ \left( (M_f-z)^{-1}u_n \right) }{\left( (M_{\conj{f}} -\conj{z})^{-1} v_n \right)} $. Then we have that 
$$
  \norme{A(z) - A_N(z)}
    =   \norme{\sum_{n =N+1}^\infty s_n \tens{ \left( (M_f-z)^{-1}u_n \right) }{\left( (M_{\conj{f}} -\conj{z})^{-1} v_n \right)}}
    \le \sum_{n=N+1}^\infty  s_n.
$$
  The last term is the tail of a convergent series, so it goes to $0$ as $N$ goes to infinity. So $A(z) $ is a limit of finite rank operators, hence it is a compact operator.
  
  Let $z_1, z_2 \in \Gamma$. Then
$$
  \norme{A(z_1) - A(z_2)}
    \le \norme{A(z_1) - A_N(z_1)} + \norme{A_N(z_1) - A_N(z_2)} + \norme{A_N(z_2) - A(z_2)}.
$$  
  The quantities on the right hand side are small if $ N$ is big enough and $z_1$ is close enough of $z_2$. So $A : \Gamma \rightarrow \K{H} $ is a continuous application. Hence $M_f + \sum_{n \in \N} s_n \tens{u_n}{v_n} $ has a non trivial hyperinvariant subspace.
\end{example}

  Now we give a version of Corollary \ref{CoroNormalPlusKompact} for compact perturbations of diagonal operators.

\begin{corollary}
  \label{CoroDiagonalPlusKompact}
  Let $(e_n)_{n \in \N} $ be an orthonormal basis of the Hilbert space $H$. Let $D = \sum_{n \in \N} \lambda_n \tens{e_n}{e_n} $ be a bounded diagonal operator on $H$. Let $(s_n)_{n \in \N}$ be a sequence of positive real numbers such that $\lim_{n \rightarrow \infty} s_n = 0$. Let $(u_n)_{n \in \N}, (v_n)_{n \in \N} $ be two orthonormal families of vectors in $H$. We denote $ K = \sum_{n \in \N} s_n \tens{u_n}{v_n} $.
   Suppose that there exists a rectifiable piecewise smooth Jordan curve $\Gamma$ such that
  \begin{itemize}
    \item[1.] There exist two accumulation points $a,b$ of eigenvalues of $D$ such that $a$ is in the interior of $\Gamma$ and $b$ is in the exterior of $\Gamma$,
    \item[2.] $\Gamma \cap \Specp{D} = \emptyset $,
    \item[3.] For all $n \in \N$, $z \in \Gamma$, we have that $u_n \in \Range(D -z)$ and $v_n \in \Range(D-z)^* $, 
    \item[4.] Denote by $A(z)$ the (possibly unbounded) operator $A(z) = \sum_{n \in \N} s_n \tens{ \left( (D-z)^{-1}u_n \right) }{\left( (D^* -\conj{z})^{-1} v_n \right)}$. For all $z \in \Gamma$, we suppose that $A(z)$ is a compact operator, and $A : \Gamma \rightarrow \K{H} $ is a continuous application. 
  \end{itemize}   
  Then the operator $T = D + K$ has a non trivial hyperinvariant subspace.
\end{corollary}

\begin{proof}
  Let $\Omega = \N $. Let $\mu = \sum_{n \in \N} \frac{1}{2^n} \delta_{\lbrace n \rbrace} $, with $\delta_{\lbrace n \rbrace} $ being the Dirac measure at the point $ \lbrace n \rbrace $. Let $f: \N \rightarrow \C $ be defined by $f(n)=\lambda_n $. Then $D$ is unitarily equivalent to $M_f$, the multiplication by $f$ on $\Lp{2}(\Omega,\mu)$. As $a$ and $b$ are accumulation points of eigenvalues of $D$, we have that $ a,b \in \Spece{D} = \Spece{M_f} $. As $\Gamma \cap \Specp{D} = \emptyset $, we have that $f^{-1}(\Gamma) = \emptyset $ so $\mu(f^{-1}(\Gamma)) = 0$. By Corollary \ref{CoroNormalPlusKompact}, $D + K$ has a non trivial hyperinvariant subspace.
\end{proof}

\section{Consequences for compact perturbations of diagonal operators on a Hilbert space: proof of Theorem \ref{ThDiagPlusKompactSummabilityCondition}}

  The goal of this section is to prove Theorem \ref{ThDiagPlusKompactSummabilityCondition}.
  We will need some material before proving Theorem \ref{ThDiagPlusKompactSummabilityCondition}. First we will need a modified version of Lemma 2.1 of \cite{Fang_Xia_2012}.
 

\begin{lemma}
  \label{Lem21}
  Let $(\lambda_k)_{k \in \N}$ be a bounded sequence of complex numbers, and let $(\alpha_{n,k})_{n,k \in \N} $ be a sequence of complex numbers such that
  $$
    \sum_{n \in \N} \sum_{k \in \N} \abs{\alpha_{n,k}} < \infty.
  $$
  Then for almost every $x \in \R$ we have that
  $$
    \sum_{n \in \N} \sum_{k \in \N} \frac{\abs{\alpha_{n,k}}^2}{\abs{\Reel{\lambda_k}-x}^2} < \infty.
  $$
\end{lemma}

\begin{proof}
  Suppose that $\sum_{n \in \N} \sum_{k \in \N} \abs{\alpha_{n,k}} < \infty $. Then for every $\varepsilon > 0$ there exists $\delta >0 $ such that $ 2\delta \sum_{n \in \N} \sum_{k \in \N} \abs{\alpha_{n,k}} < \epsilon $. We denote by $ I_{n,k}$ the interval $ [\Reel{\lambda_k} - \delta \alpha_{n,k}, \Reel{\lambda_k} + \delta \alpha_{n,k} ] $, and we define the functions $f_{n,k} $ on $\R$ by 
  $$
  f_{n,k}(x) = \frac{\abs{\alpha_{n,k}}^2}{\abs{\Reel{\lambda_k}-x}^2} \ind{\R \setminus I_{n,k}}(x).
  $$ 
  We have that
$$
  \integral{\R}{f_{n,k}(x)}{x}
    = \integral{\R \setminus I_{n,k}}{\frac{\abs{\alpha_{n,k}}^2}{\abs{\Reel{\lambda_k}-x}^2}}{x} 
    = \abs{\alpha_{n,k}}^2 \frac{2}{\delta \abs{\alpha_{n,k}}} 
    = \frac{2 \abs{\alpha_{n,k}}}{\delta }.
$$

Let us denote by $F$ the function $F(x) = \sum_{n \in \N} \sum_{k \in \N} f_{n,k}(x) $. As the functions $f_{n,k}$ are non negative functions, using Beppo-Levi Theorem we have that
$$
  \integral{\R}{F(x)}{x} 
    = \sum_{k \in \N} \sum_{n \in \N} \integral{\R}{f_{n,k}(x)}{x}
    = \frac{2}{\delta } \sum_{k \in \N} \sum_{n \in \N} \abs{\alpha_{n,k}} < \infty.
$$ 
  So $F$ belongs to $\Lun$, and for almost every $x \in \R$, we have $F(x) < \infty $.
  Denote by $\Lambda$ the set
$$
  \Lambda = \left\lbrace x \in \R, \, \sum_{k \in \N} \sum_{n \in \N} \frac{\abs{\alpha_{n,k}}^2}{\abs{\Reel{\lambda_k}-x}^2} = \infty \right\rbrace.
$$
Obviously we have that
$$
  \Lambda \subset \left( \bigcup_{k,n \in \N} I_{n,k} \right) \cup \lbrace x \in \R, F(x) = \infty \rbrace.
$$
Using the additivity of the Lebesgue measure we get that
\begin{align*}
  \m(\Lambda) 
    &\le \sum_{k \in \N} \sum_{n \in \N} \m(I_{n,k}) + \m(\lbrace x \in \R, F(x) = \infty \rbrace) \\
    &=   2\delta \sum_{k \in \N} \sum_{n \in \N} \abs{\alpha_{n,k}} + 0 \\
    &\le \varepsilon .
\end{align*}
  As $\varepsilon$ was chosen arbitrarily, we eventually get that $\m(\Lambda) =0$.
\end{proof}

\begin{lemma}
  \label{LemSumFiniPP}
  Suppose that conditions (\ref{Cond1}) and (\ref{Cond2}) of Theorem \ref{ThDiagPlusKompactSummabilityCondition} are satisfied, then for almost every $x \in \R$, we have that
$$
  \sum_{k\in \N} \sum_{n\in \N} \frac{\abs{ a_n \ps{u_n}{e_k}}^2}{\abs{\Reel{\lambda_k}-x}^2} < \infty , \quad
  \sum_{n\in \N} \sum_{j\in \N} \frac{\abs{ b_n \ps{e_j}{v_n}}^2}{\abs{\Reel{\lambda_k}-x}^2} < \infty 
$$
\end{lemma}

\begin{proof}
  This is a direct consequence of Lemma \ref{Lem21}.
\end{proof}
  
  In order to use Theorem \ref{ThMainTheorem}, we need to define a Jordan curve $\Gamma$ that will split the eigenvalues of $D$ in two parts. Then we will need to check whether $A(z)$ has the properties required on $\Gamma$. First we write $A_1(z) = \sum_{n \in \N} a_n \tens{ \left( (D-z)^{-1}u_n \right) }{e_n}$ and $A_2(z) = \sum_{n \in \N} b_n \tens{e_n}{\left( (D^* -\conj{z})^{-1} v_n \right)}$. Note that if $A_1$ and $A_2$ has the properties required by Theorem \ref{ThMainTheorem}, then
  \begin{align*}
  A_1(z)A_2(z) 
    &=  \left( \sum_{n \in \N} a_n \tens{ \left( (D-z)^{-1}u_n \right) }{e_n} \right)\left( \sum_{n \in \N} b_n \tens{e_n}{\left( (D^* -\conj{z})^{-1} v_n \right)} \right) \\
    &=  \sum_{n \in \N} s_n \tens{ \left( (D-z)^{-1}u_n \right) }{\left( (D^* -\conj{z})^{-1} v_n \right)} \\
    &= A(z), 
  \end{align*} 
  and $A(z)$ has the required properties.
  Now we will need some estimates on $\norme{A_1(z)}$ and $\norme{A_2(z)}$. After that we will be able to draw the Jordan curve $\Gamma$ that we need.  
  
\begin{lemma}
  \label{LemMajDzA1}
  Let $z \in \C \setminus \lbrace \lambda_k, k \in \N \rbrace $. We denote $x= \Reel{z}$. Suppose that condition (\ref{Cond1}) of Theorem \ref{ThDiagPlusKompactSummabilityCondition} is satisfied. Then for almost every $x \in \R \setminus \lbrace \Reel{\lambda_k}, k \in \N \rbrace $, $A_1(z)$ is a bounded operator and we have
  $$
    \norme{A_1(z)}^2 \le \sum_{k\in \N} \sum_{n\in \N} \frac{\abs{ a_n \ps{u_n}{e_k}}^2}{\abs{\Reel{\lambda_k}-x}^2}.
  $$
\end{lemma}

\begin{proof}
  Let $z \in \C \setminus \lbrace \lambda_k, k \in \N \rbrace $. Note that $\abs{\Reel{\lambda_k-z}} \le \abs{\lambda_k-z} $. So we have that
$$
  \sum_{k\in \N} \sum_{n\in \N} \frac{\abs{ a_n \ps{u_n}{e_k}}^2}{\abs{\lambda_k-z}^2}
    \le \sum_{k\in \N} \sum_{n\in \N} \frac{\abs{ a_n \ps{u_n}{e_k}}^2}{\abs{\Reel{\lambda_k}-x}^2}.
$$
  Let $h \in H$. Using Cauchy-Schwartz inequality we get that
\begin{align*}
  \norme{(A_1(z)) (h)}^2
    &=   \sum_{k\in \N} \abs{\sum_{n\in \N}\frac{ a_n \ps{h}{e_n} \ps{u_n}{e_k}}{\lambda_k-z}}^2 \\
    &\le \sum_{k \in \N} \norme{\sum_{n\in \N}\frac{ a_n \ps{u_n}{e_k}}{\lambda_k-z} e_n}^2 \norme{\sum_{n \in \N} \ps{e_n}{h}e_n}^2 \\
    &= \sum_{k \in \N} \sum_{n\in \N} \abs{\frac{ a_n \ps{u_n}{e_k}}{\lambda_k-z}}^2 \norme{h}^2.
\end{align*}
  Hence the inequality of Lemma \ref{LemMajDzA1} holds. We used the condition (\ref{Cond1}) in Cauchy Schwartz inequality to ensure that $ \left(\frac{a_n \ps{u_n}{e_k}}{\lambda_k-z}\right)_{n\in \N} $ is a square summable sequence.
\end{proof}

  Similarly, one can prove the following lemma.

\begin{lemma}
  \label{LemMajA2Dz}
  Let $z \in \C \setminus \lbrace \lambda_k, k \in \N \rbrace $. We denote $x= \Reel{z}$. Suppose that condition (\ref{Cond2}) of Theorem \ref{ThDiagPlusKompactSummabilityCondition} is satisfied. Then for almost every $x \in \R \setminus \lbrace \Reel{\lambda_k}, k \in \N \rbrace $, $A_2(z)$ is bounded and we have
  $$
    \norme{A_2(z)}^2 \le \sum_{n\in \N} \sum_{j\in \N} \frac{\abs{ b_n \ps{e_j}{v_n}}^2}{\abs{\Reel{\lambda_k}-x}^2}.
  $$
\end{lemma}

\begin{lemma}
  \label{LemCompactContinu}
  Suppose that conditions (\ref{Cond1}) and (\ref{Cond2}) of Theorem \ref{ThDiagPlusKompactSummabilityCondition} are satisfied, then for almost every $ x_0 \in \R \setminus \lbrace \Reel{\lambda_k}, k \in \N \rbrace$, for every $ z \in s_0 =\lbrace z = x_0 + \I y, y \in \R \rbrace $, we have that $A_1(z)$ and $A_2(z)$ are compact operators. Moreover the maps $A_1 : s_0 \rightarrow \K{H} $ and $A_2 : s_0 \rightarrow \K{H} $ are continuous.
\end{lemma}

\begin{proof}
  First note that conditions (\ref{Cond1}) and (\ref{Cond2}) and Lemmas \ref{LemMajDzA1} and \ref{LemMajA2Dz} give us that the operators $A_1(z)$ and $A_2(z)$ are bounded for almost every $x_0$. Let $E_N$ be the orthogonal projection of H onto the subspace generated by $ e_0, e_1, \dots, e_N$. Then we have that
$$
  E_N A_1(z) = \sum_{k \le N} \sum_{n \in \N} \frac{a_n \ps{u_n}{e_k}}{\lambda_k -z}\tens{e_k}{e_n}.
$$
Note that $ E_N A_1(z) $ has finite rank. So we get that
$$
A_1(z) - E_N A_1(z) = \sum_{k > N} \sum_{n \in \N} \frac{a_n \ps{u_n}{e_k}}{\lambda_k -z}\tens{e_k}{e_n}.
$$
Using Lemma \ref{LemMajDzA1}, we get that
$$
  \norme{A_1(z) - E_N A_1(z)} \le \sum_{k > N} \sum_{n\in \N} \frac{\abs{ a_n \ps{u_n}{e_k}}^2}{\abs{\Reel{\lambda_k}-x}^2}.
$$
According to Lemma \ref{LemSumFiniPP}, the right term is the tail of a convergent series for almost every $x_0 \in \R $, so it goes to zero as $N$ goes to infinity. Therefore $A_1(z)$ is a uniform limit of finite rank operators, so it is a compact operator.

Now take $z_1, z_2 \in s_0 $. Thanks to the triangular inequality we get that 
\begin{align*}
\norme{A_1(z_1) - A_1(z_2)} 
  \le & \norme{A_1(z_1) - E_N A_1(z_1)} \\
      & + \norme{E_N A_1(z_1) - E_N A_1(z_2)} \\
      & + \norme{E_N A_1(z_2) -  A_1(z_2)}.
\end{align*}
We can fix $N \in \N$ big enough, such that the norms $\norme{A_1(z_1) - E_N A_1(z_1)}$ and \\
$\norme{E_N A_1(z_2) -  A_1(z_2)}$ are small. Now a simple computation give that
$$
  E_N A_1(z_1) - E_N A_1(z_2) = \left( \sum_{k=1}^N \left(\frac{1}{\lambda_k - z_1} - \frac{1}{\lambda_k - z_2} \right) \tens{e_k}{e_k} \right) \left( \sum_{n \in \N} a_n \tens{u_n}{e_n} \right).
$$
So we have that
\begin{align*}
  \norme{E_N A_1(z_1) - E_N A_1(z_2)}
    &\le \max_{ k=1, \dots, N } \abs{ \frac{1}{\lambda_k - z_1} - \frac{1}{\lambda_k - z_2} } \norme{\sum_{n \in \N} a_n \tens{u_n}{e_n}}.
\end{align*}

  Note that $ \norme{\sum_{n \in \N} a_n \tens{u_n}{e_n}} $ does not depend on $z_1, z_2 $. Remember that for every $k \in \N $, $x_0 \ne \Reel{\lambda_k}$, so the function $f_k : \R \rightarrow \C $ defined by $f(y) = \frac{1}{\lambda_k - x_0 - \I y} $ is continuous. So $\max_{ k=1, \dots, N } \abs{\frac{1}{\lambda_k - z_1} - \frac{1}{\lambda_k - z_2} }$ is small when $z_1$ is close to $z_2$.
We deduce that $\norme{ E_N A_1(z_1) - E_N A_1(z_2)}$ is small when $z_1$ is close to $z_2$. It follows that the maps $A_1 : s_0 \rightarrow \K{H} $ is continuous. The same proof works for the map  $A_2: s_0 \rightarrow \K{H}$.
\end{proof}

  Note that if $A_1(z)$ and $A_2(z)$ satisfy condition 3 and 4 of Theorem \ref{ThMainTheorem}, so does $A(z) = A_1(z)A_2(z) $.
  
\begin{proof}[Proof of Theorem \ref{ThDiagPlusKompactSummabilityCondition}]
  Denote $\rho$ the spectral radius of $D$. If $\Spece{D} = \lbrace \lambda \rbrace$, then there exists a compact operator $K_e$ such that $D = \lambda I + K_e  $. So $T=D+K = \lambda I + K_e + K$ is a compact perturbation of a scalar operator, and Lomonosov Theorem (see \cite[Theorem 6.1.2]{Chalendar_Partington_2011}) gives the existence of a non trivial hyperinvariant subspace.
  
  Suppose that $\Spece{D}$ contain a least two points $a$ and $b$. Considering if necessary a certain rotation $e^{\I \theta} D $ of $D$ we can assume that $\Reel{a} < \Reel{b} $. 
  By Lemma \ref{LemCompactContinu}, for almost every $x_0 \in ]\Reel{a},\Reel{b}[  \setminus \lbrace \Reel{\lambda_k}, k \in \N \rbrace$, denote $s_0 = \lbrace x_0 + \I y, y \in [-\rho-1,\rho+1] \rbrace $, we have that $A: s_0 \rightarrow \K{H}$ is a well defined and continuous application. Denote 
\begin{align*}
  s_1 &= \lbrace x + \I(\rho+1), x \in [x_0-\rho-1,x_0] \rbrace \\
  s_2 &= \lbrace x_0-\rho-1 + \I y, y \in [-\rho-1,\rho+1] \rbrace \\ 
  s_3 &= \lbrace x - \I(\rho+1), x \in [x_0-\rho-1,x_0] \rbrace.
\end{align*} 
Note that $\left( s_1 \cup s_2 \cup s_3 \right) \cap \Spec{D} = \emptyset $. So for all $z \in s_1 \cup s_2 \cup s_3$, $ (D-z)^{-1} $ is a bounded operator. So we have that
\begin{align*}
  A(z) &= \sum_{n \in \N} s_n \tens{ \left( (D-z)^{-1}u_n \right) }{\left( (D^* -\conj{z})^{-1} v_n \right)} \\
       &=  (D-z)^{-1} \left(\sum_{n \in \N} s_n \tens{u_n}{v_n}\right) (D-z)^{-1} \\
       &=  (D-z)^{-1} K (D-z)^{-1}.
\end{align*}
  Obviously $A: s_1 \cup s_2 \cup s_3 \rightarrow \K{H} $ is well defined and continuous. Denote $\Gamma = s_0 \cup s_1 \cup s_2 \cup s_3 $. As $A: s_0 \rightarrow \K{H} $ is also continuous and $ s_0 \cap \left( s_1 \cup s_2 \cup s_3 \right) = \lbrace x_0 - \I (\rho+1), x_0 + \I (\rho+1) \rbrace \ne \emptyset$, we have that $ A : \Gamma \rightarrow \K{H}$ is continuous.  Finally an application of Theorem \ref{ThMainTheorem} completes the proof.
\end{proof}

\section*{Acknowledgments}
  I would like to thank Sophie Grivaux for several discussions and for her help to improve this paper. I would like also to thank the referees for the careful reading of the manuscript, and their suggestions to improve the presentation of the paper.

\bibliographystyle{alpha}
\bibliography{Biblio}

\end{document}